\documentclass[preprint,11pt]{elsarticle}
\makeatletter
    
    \newcommand{\Rmnum}[1]{\expandafter\@slowromancap\romannumeral #1@}
\makeatother

\makeatletter
\def\ps@pprintTitle{%
  \let\@oddhead\@empty
  \let\@evenhead\@empty
  \def\@oddfoot{\reset@font\hfil\thepage\hfil}
  \let\@evenfoot\@oddfoot
}
\makeatother
\usepackage{}
\usepackage{amssymb}
\usepackage{amsthm}
\usepackage{amsmath}
\usepackage{amscd}
\usepackage{graphicx}
\usepackage{epstopdf}
\usepackage{lineno}
\usepackage{txfonts}
\usepackage{mathrsfs}
\usepackage{amsfonts}
\usepackage[all]{xy}
\usepackage{tikz}
\usetikzlibrary{arrows}
\biboptions{numbers,sort&compress}
\newtheorem{theorem}{Theorem}[section]
\newtheorem{lemma}{Lemma}[section]
\newtheorem{corollary}{Corollary}[section]

\newtheorem{example}{Example}[section]
\newtheorem{remark}{Remark}[section]

\journal{Elsevier}

\begin{document}
\newcommand{\T}{\mathbb{T}}
\newcommand{\R}{\mathbb{R}}
\newcommand{\Q}{\mathbb{Q}}
\newcommand{\N}{\mathbb{N}}
\newcommand{\Z}{\mathbb{Z}}
\newcommand{\tx}[1]{\quad\mbox{#1}\quad}
\renewcommand{\theequation}%
{\arabic{section}.\arabic{equation}}
\numberwithin{equation}{section}


\begin{frontmatter}

\title{New oscillation criteria for third-order half-linear advanced differential equations
\tnoteref {t1}}\tnotetext[t1]{This work was supported by Youth Program of National Natural Science
Foundation of China (61304008)
and Youth Program of  Natural Science Foundation of Shandong Province  (ZR2013FQ033).}

\author{Jianli Yao \corref{t1}}
\cortext[t1]{Corresponding Author. Tel.: +86 18866865936 }
\ead{yaojianli@sdjzu.edu.cn}
\author{Xiaoping Zhang}
\ead{zxp@sdjzu.edu.cn}
\author{Jiangbo Yu}
\ead{jbyu@sdjzu.edu.cn}
\address{School of Science, Shandong Jianzhu
University, Ji'nan 250101, P.R.China}

\begin{abstract}
\quad The theme of this article is to provide some sufficient conditions for the asymptotic property  and oscillation of all solutions of third-order half-linear differential equations with advanced argument of the form
$$\left(r_{2}(t)\left(\left(r_{1}(t)\left(y'(t)\right)^{\alpha}\right)'\right)^{\beta}\right)'
+q(t)y^{\gamma}\left(\sigma(t)\right)=0,\ t\geq t_{0}>0,$$
where $\int^{\infty}r_{1}^{-\frac{1}{\alpha}}(s)\text{d}s<\infty$ and $\int^{\infty}r_{2}^{-\frac{1}{\beta}}(s)\text{d}s<\infty$.
The criteria in this paper improve and complement some existing ones. The results are illustrated by two
Euler-type differential equations.
\end{abstract}

\end{frontmatter}

\vspace{.4cm} \noindent{\bf 1. Introduction}
 \setcounter{section}{1}
 \setcounter{equation}{0}

In this paper, we consider the oscillatory and asymptotic behavior of solutions to the  third-order half-linear advanced differential equations of the form
\begin{equation}
\left(r_{2}(t)\left(\left(r_{1}(t)\left(y'(t)\right)^{\alpha}\right)'\right)^{\beta}\right)'
+q(t)y^{\gamma}\left(\sigma(t)\right)=0,\ t\geq t_{0}>0.
\end{equation}
Throughout the whole paper, we assume that

$(H_{1})$\ $\alpha$, $\beta$ and $\gamma$ are quotient of odd positive integers;

 $(H_{2})$\ the functions $r_{1},r_{2}\in C\left([t_{0},\infty),(0,\infty)\right)$ are of noncanonical
 type (see Trench [2]), that is,
 $$
 \pi_{1}\left(t_{0}\right):=\int_{t_{0}}^{\infty}r_{1}^{-\frac{1}{\alpha}}(s)\text{d}s<\infty,
 \ \pi_{2}\left(t_{0}\right):=\int_{t_{0}}^{\infty}r_{2}^{-\frac{1}{\beta}}(s)\text{d}s<\infty;
 $$

 $(H_{3})$\  $q\in C\left([t_{0},\infty),[0,\infty)\right)$ does not vanish eventually;

 $(H_{4})$\ $\sigma\in C^{1}\left([t_{0},\infty),(0,\infty)\right)$, $\sigma(t)\geq t$, $\sigma'(t)\geq 0$
 for all $t\geq t_{0}$.

 By a solution of $Eq.$ $(1.1)$, we mean a nontrivial real valued function  $y(t)\in C\left(\big[T_{x},\infty\big), \mathbb{R}\right)$, $T_{x}\geq t_{0}$,  which has the property that
$y$, $r_{1}\left(y'\right)^{\alpha}$,
$r_{2}\left(\left(r_{1}\left(y'\right)^{\alpha}\right)'\right)^{\beta}$
are continuous and differentiable for all $t\in \big[T_{x},\infty\big)$, and satisfies $(1.1)$ on $\big[T_{x},\infty\big)$.
We only need to consider those solutions of $(1.1)$ which exist on some half-line  $\big[T_{x},\infty\big)$ and satisfy the condition
$$ sup\{|y(t)|: T\leq t<\infty\}>0 $$
for any $T\geq T_{x}$.
In the sequel, we assume that $(1.1)$ possess such  solutions.

As is customary, a solution $y(t)$ of $(1.1)$ is called oscillatory if it has arbitrary large zeros on  $\big[T_{x},\infty\big)$.
Otherwise, it is called nonoscillatory. The equation $(1.1)$ is said to be oscillatory if all its solutions oscillate.

Following classical results of Kiguradze and Kondrat'ev ($[3]$), we say that $(1.1)$ has property $A$ if any solution $y$ of $(1.1)$ is either oscillatory or satisfies $lim_{t\rightarrow\infty} y(t)=0$. Instead of calling property $A$, some authors say that equation $(1.1)$ is almost oscillatory.

For the sake of brevity, we define the operators
$$L_{0}y=y, L_{1}y=r_{1}\left(y'\right)^{\alpha},
L_{2}y=r_{2}\left(\left(r_{1}\left(y'\right)^{\alpha}\right)'\right)^{\beta},
L_{3}y=\left(r_{2}\left(\left(r_{1}\left(y'\right)^{\alpha}\right)'\right)^{\beta}\right)'.$$

Also, we use the symbols $\uparrow$ and $\downarrow$ to indicate whether the function is nondecreasing
and nonincreasing, respectively.

\vspace{.4cm} \noindent{\bf 2. Main results}
 \setcounter{section}{2}
 \setcounter{equation}{0}

As usual, all functional inequalities considered in this paper are supposed to hold eventually,
that is, they are satisfied for all $t$ large enough.

Without loss of generality, we need only to consider eventually positive solutions of $(1.1)$,
since if $y$ satisfies $(1.1)$, so does $-y$.

The following lemma on the structure of possible nonoscillatory solutions of $(1.1)$ plays a crucial role
in the proofs of the main results.

\begin{lemma}\label{2.1}
\ Assume $(H_{1})-(H_{4})$,  and that $y$ is an eventually positive solution of  $Eq.$ $(1.1)$.
Then there exist $t_{1}\in\big[t_{0}, \infty\big)$ such that $y$ eventually belongs to one  of the following classes:
$$ S_{1}=\big\{ y:\  y>0,\ \  L_{1}y<0, \ \  L_{2}y<0,\ \  L_{3}y<0\big\};$$
$$ S_{2}=\big\{ y:\  y>0,\ \  L_{1}y<0, \ \  L_{2}y>0,\ \  L_{3}y<0\big\};$$
$$ S_{3}=\big\{ y:\  y>0,\ \  L_{1}y>0, \ \  L_{2}y>0,\ \  L_{3}y<0\big\};$$
$$ S_{4}=\big\{ y:\  y>0,\ \  L_{1}y>0, \ \  L_{2}y<0,\ \  L_{3}y<0\big\},$$
for $t\geq t_{1}$.

\end{lemma}

\begin{proof}
\ The proof is straightforward and hence is omitted.
\end{proof}

Now, we will establish  one-condition criteria of property $A$ of $(1.1)$.

\begin{theorem}\label{thm2.1}
\ Assume $(H_{1})-(H_{4})$. If
\begin{equation}\label{2.1}
\int_{t_{0}}^{\infty}r_{1}^{-\frac{1}{\alpha}}(v)\left(\int_{t_{0}}^{v}r_{2}^{-\frac{1}{\beta}}(u)
\left(\int_{t_{0}}^{u}q(s)\text{d}s\right)^{\frac{1}{\beta}}\text{d}u\right)^{\frac{1}{\alpha}}\text{d}v = \infty ,
\end{equation}
then $(1.1)$ has property $A$.
\end{theorem}

\begin{proof}
\ First of all, it is important to note that if  $(H_{2})$ and $(2.1)$ hold, then
\begin{equation}\label{2.2}
\int_{t_{0}}^{\infty}r_{2}^{-\frac{1}{\beta}}(u)
\left(\int_{t_{0}}^{u}q(s)\text{d}s\right)^{\frac{1}{\beta}}\text{d}u
= \left(\int_{t_{0}}^{\infty}q(s)\text{d}s\right)^{\frac{1}{\beta}}= \infty ,
\end{equation}
i.e.,
\begin{equation}\label{2.3}
\int_{t_{0}}^{\infty}q(s)\text{d}s= \infty.
\end{equation}

Now, suppose on the contrary that $y$ is a nonoscillatory solution of $(1.1)$ on $\big[t_{0},\infty\big)$.
Without loss of generality, we may assume that $t_{1}\geq t_{0}$ such that $y(t)>0$ and $ y\left(\sigma(t)\right)>0$
for $t\geq t_{1}$. By $Lemma$ $ 2.1$, we know that $y$ eventually belongs to one  of the four classes in $Lemma$ $ 2.1$.
 We will consider each of them separately.

Assume $y\in S_{1}$. Then from $L_{1}y<0$, that is, $r_{1}\left(y'\right)^{\alpha}<0$, we see that $y'<0$ and $y$ is decreasing. On the other words, there exists a finite constant $\ell\geq 0$ such that $lim_{t\rightarrow\infty} y(t)=\ell$.
Obviously, $lim_{t\rightarrow\infty} y\left(\sigma(t)\right)=\ell$, too.

We claim that $\ell=0$. Assume on the contrary that $\ell>0$. Then there exists $t_{2}\geq t_{1}$ such that $y(t)\geq y\left(\sigma(t)\right)\geq \ell$ for $t\geq t_{2}$.
Thus,
\begin{equation}\label{2.4}
-L_{3}y(t)=q(t)y^{\gamma}\left(\sigma(t)\right)\geq \ell^{\gamma}\cdot q(t),
\end{equation}
for $t\geq t_{2}$. Integrating $(2.4)$ from $t_{2}$ to $t$, we have
$$-L_{2}y(t)\geq -L_{2}y\left(t_{2}\right)+\ell^{\gamma}\int_{t_{2}}^{t}q(s)\text{d}s
\geq \ell^{\gamma}\int_{t_{2}}^{t}q(s)\text{d}s.$$
Therefore,
\begin{equation}\label{2.5}
-\left(L_{1}y\right)'(t)\geq \ell^{\frac{\gamma}{\beta}}r_{2}^{-\frac{1}{\beta}}(t)
\left(\int_{t_{2}}^{t}q(s)\text{d}s\right)^{\frac{1}{\beta}}.
\end{equation}
Integrating $(2.5)$ again from $t_{2}$ to $t$, we have
$$\aligned -L_{1}y(t)
&\geq -L_{1}y\left(t_{2}\right)+\ell^{\frac{\gamma}{\beta}}\int_{t_{2}}^{t}r_{2}^{-\frac{1}{\beta}}(u)
\left(\int_{t_{2}}^{u}q(s)\text{d}s\right)^{\frac{1}{\beta}}\text{d}u\\
&\geq \ell^{\frac{\gamma}{\beta}}\int_{t_{2}}^{t}r_{2}^{-\frac{1}{\beta}}(u)
\left(\int_{t_{2}}^{u}q(s)\text{d}s\right)^{\frac{1}{\beta}}\text{d}u,
\endaligned$$
that is,
\begin{equation}\label{2.6}
-y'(t)\geq \ell^{\frac{\gamma}{\alpha\beta}}r_{1}^{-\frac{1}{\alpha}}(t)\left(\int_{t_{2}}^{t}r_{2}^{-\frac{1}{\beta}}(u)
\left(\int_{t_{2}}^{u}q(s)\text{d}s\right)^{\frac{1}{\beta}}\text{d}u\right)^{\frac{1}{\alpha}}.
\end{equation}
Integrating $(2.6)$  from $t_{2}$ to $t$ for the last time, and taking account of $(2.1)$, we have
$$y(t)\leq y\left(t_{2}\right)-\ell^{\frac{\gamma}{\alpha\beta}}
\int_{t_{2}}^{t}r_{1}^{-\frac{1}{\alpha}}(v)\left(\int_{t_{2}}^{v}r_{2}^{-\frac{1}{\beta}}(u)
\left(\int_{t_{2}}^{u}q(s)\text{d}s\right)^{\frac{1}{\beta}}\text{d}u\right)^{\frac{1}{\alpha}}\text{d}v \rightarrow -\infty ,$$
as $t\rightarrow\infty$, which contradicts to the positivity of $y$. Thus,  $lim_{t\rightarrow\infty} y(t)=0$.

Assume $y\in S_{2}$. Proceeding the same steps above, we arrive at $(2.4)$.  Integrating $(2.4)$ from $t_{2}$ to $t$, we have
\begin{equation}\label{2.7}
L_{2}y(t)
\leq L_{2}y\left(t_{2}\right)-\ell^{\gamma}\int_{t_{2}}^{t}q(s)\text{d}s\rightarrow -\infty ,\ \ (t\rightarrow\infty),
\end{equation}
where we used $(2.3)$. This contradicts to the positivity of $L_{2}y$ and thus, $lim_{t\rightarrow\infty} y(t)=0$.

Assume $y\in S_{3}$. We define a function
$$w(t):=\frac{L_{2}y(t)}{y^{\gamma}\left(\sigma(t)\right)},\ \  t\geq t_{2}.$$
Obviously, $w(t)$ is positive for $t\geq t_{2}$. Using $(1.1)$, we obtain
$$\aligned
w'(t)
&=\frac{\left(L_{2}y\right)'(t)}{y^{\gamma}\left(\sigma(t)\right)}
-\frac{L_{2}y(t)\cdot\gamma\cdot y^{\gamma-1}\left(\sigma(t)\right)\cdot y'\left(\sigma(t)\right)\cdot\sigma'(t)}{y^{2\gamma}\left(\sigma(t)\right)}\\
&=\frac{L_{3}y(t)}{y^{\gamma}\left(\sigma(t)\right)}
-\gamma\frac{L_{2}y(t)\cdot y'\left(\sigma(t)\right)\cdot\sigma'(t)}{y^{\gamma +1}\left(\sigma(t)\right)}\\
&\leq \frac{L_{3}y(t)}{y^{\gamma}\left(\sigma(t)\right)}=-q(t).
\endaligned
$$
Integrating the above inequality from $t_{2}$ to $t$, and taking  $(2.3)$ into account, we have
$$w(t)\leq w\left(t_{2}\right)-\int_{t_{2}}^{t}q(s)\text{d}s\rightarrow -\infty ,\ \  (t\rightarrow\infty).$$
This contradicts to the positivity of $w$. Hence, $ S_{3}={\O}$.

Assume $y\in S_{4}$. Considering that $y$ is increasing, and integrating $(1.1)$ from $t_{2}$ to $t$,
we obtain
$$-L_{2}y(t)=-L_{2}y\left(t_{2}\right)+\int_{t_{2}}^{t}q(s)y^{\gamma}\left(\sigma(s)\right)\text{d}s
\geq y^{\gamma}\left(\sigma(t_{2})\right)\int_{t_{2}}^{t}q(s)\text{d}s,$$
i.e.,
\begin{equation}\label{2.8}
-\left(L_{1}y\right)'(t)
\geq k^{\frac{1}{\beta}}r_{2}^{-\frac{1}{\beta}}(t)
\left(\int_{t_{2}}^{t}q(s)\text{d}s\right)^{\frac{1}{\beta}},
\end{equation}
where $k:=y^{\gamma}\left(\sigma\left(t_{2}\right)\right)$. Integrating $(2.8)$ from $t_{2}$ to $t$ and using $(2.2)$, we have
\begin{equation}\label{2.9}
\aligned
L_{1}y(t)
&\leq L_{1}y\left(t_{2}\right)
-k^{\frac{1}{\beta}}\int_{t_{2}}^{t}r_{2}^{-\frac{1}{\beta}}(u)
\left(\int_{t_{2}}^{u}q(s)\text{d}s\right)^{\frac{1}{\beta}}\text{d}u\\
&\rightarrow -\infty ,\ \ (t\rightarrow\infty).
\endaligned
\end{equation}
This is a contradiction to the positivity of $L_{1}y$. Thus, $ S_{4}={\O}$. The proof is complete.
\end{proof}

\begin{remark}
\ It is clear that any nonoscillatory solution in $Theorem$ $ 2.1$ eventually belongs to  either
$S_{1}$ or $S_{2}$ in $Lemma$ $ 2.1$, that is, $ S_{3}=S_{4}={\O}$.
\end{remark}

Next, we formulate some additional information about the monotonicity of solutions in $S_{2}$ or $S_{1}$.

\begin{lemma}\label{2.2}
\ Assume $(H_{1})-(H_{4})$. Let $y\in S_{2}$  in $Lemma$ $ 2.1$ on $\big[t_{1}, \infty\big)$ for some $t_{1}\geq t_{0}$,
and define the function
\begin{equation}\label{2.10}
\pi (t):=\int_{t}^{\infty}r_{1}^{-\frac{1}{\alpha}}(s)\pi_{2}^{\frac{1}{\alpha}}(s)\text{d}s.
\end{equation}
If
\begin{equation}\label{2.11}
\int_{t_{0}}^{\infty}q(s)\pi^{\gamma}\left(\sigma(s)\right)\text{d}s=\infty,
\end{equation}
then there exists $t_{2}\geq t_{1}$ such that
\begin{equation}\label{2.12}
\frac{y(t)}{\pi (t)}\downarrow 0
\end{equation}
for $t\geq t_{2}$.
\end{lemma}

\begin{proof}
\ Let $y\in S_{2}$  in $Lemma$ $ 2.1$ on $\big[t_{1}, \infty\big)$ for some $t_{1}\geq t_{0}$. First, we prove that
$(2.11)$ implies
\begin{equation}\label{2.13}
\lim\limits_{t\rightarrow\infty} \frac{y(t)}{\pi (t)}=0.
\end{equation}
Using I'Hospital rule, we obtain
$$\lim\limits_{t\rightarrow\infty} \frac{y(t)}{\pi (t)}
=-\left(\lim\limits_{t\rightarrow\infty}\frac{L_{1}y(t)}{\pi_{2}(t)}\right)^{\frac{1}{\alpha}}
=\left(\lim\limits_{t\rightarrow\infty}L_{2}y(t)\right)^{\frac{1}{\alpha\beta}}. $$
Taking the decreasing of $L_{2}y(t)$ into account , there exists a finite constant $\ell\geq 0$ such that $lim_{t\rightarrow\infty}L_{2}y(t)=\ell$.
We claim that $\ell=0$. If not, then $L_{2}y(t)\geq \ell>0$, and $y(t)\geq \ell^{\frac{1}{\alpha\beta}}\pi(t)$ eventually,
say for $t\geq t_{2}$ for some $t_{2}\in \big[t_{1},\infty\big)$. Using this relation in $(1.1)$, we obtain
$$-L_{3}y(t)\geq \ell^{\frac{\gamma}{\alpha\beta}}q(t)\pi^{\gamma}\left(\sigma(t)\right),\ \ t\geq t_{2}.$$
Integrating the above inequality from $t_{2}$ to $t$, we have
$$L_{2}y(t)\leq L_{2}y\left(t_{2}\right)
-\ell^{\frac{\gamma}{\alpha\beta}}\int_{t_{2}}^{t}q(s)\pi^{\gamma}\left(\sigma(s)\right)\text{d}s
\rightarrow -\infty ,\ \  (t\rightarrow\infty),$$
which is a contradiction. Thus $(2.13)$ holds and consequently, also
\begin{equation}\label{2.14}
\lim\limits_{t\rightarrow\infty} y(t)=\lim\limits_{t\rightarrow\infty} L_{1}y(t)=0
\end{equation}
due to  the decreasing property of $\pi(t)$ and $\pi_{2}(t)$, respectively. Considering
the monotonicity of $L_{2}y$ together with $(2.14)$ yields
$$-L_{1}y(t)=L_{1}y(\infty)-L_{1}y(t)=\int_{t}^{\infty}r_{2}^{-\frac{1}{\beta}}(s)
\left(L_{2}y(s)\right)^{\frac{1}{\beta}}\text{d}s
\leq \pi_{2}(t)\left(L_{2}y(t)\right)^{\frac{1}{\beta}},$$
 and hence, there exists $t_{3}\geq t_{2}$ such that
 $$\left(\frac{L_{1}y}{\pi_{2}}\right)'(t)=\frac{\left(L_{2}y(t)\right)^{\frac{1}{\beta}}\cdot\pi_{2}(t)+L_{1}y(t)}
 {r_{2}^{\frac{1}{\beta}}(t)\cdot\pi_{2}^{2}(t)}\geq 0,\  t\geq t_{3}.$$
 Then $\frac{L_{1}y}{\pi_{2}}$ is increasing on  $\big[t_{3}, \infty\big)$. Using it together with $(2.14)$ leads to
 $$y(t)=y(t)-y(\infty)=-\int_{t}^{\infty}\frac{\pi_{2}^{\frac{1}{\alpha}}(s)\left(L_{1}y(s)\right)^{\frac{1}{\alpha}}}
 {r_{1}^{\frac{1}{\alpha}}(s)\pi_{2}^{\frac{1}{\alpha}}(s)}\text{d}s
 \leq -\left(\frac{L_{1}y(t)}{\pi_{2}(t)}\right)^{\frac{1}{\alpha}}\pi(t).$$

Therefore, there exists $t_{4}\geq t_{3}$ such that
 $$\left(\frac{y}{\pi}\right)'(t)=\frac{\left(L_{1}y(t)\right)^{\frac{1}{\alpha}}\pi(t)+y(t)\pi_{2}^{\frac{1}{\alpha}}(t)}
 {r_{1}^{\frac{1}{\alpha}}(t)\pi^{2}(t)}\leq 0,\  t\geq t_{4},$$
 and we conclude that $y/\pi$ is decreasing on $\big[t_{4}, \infty\big)$. Hence, $(2.12)$ holds.
 The proof is complete.
\end{proof}

\begin{corollary}
\ Assume $(H_{1})-(H_{4})$. Let $y\in S_{2}$  in $Lemma$ $ 2.1$ on $\big[t_{1}, \infty\big)$ for some $t_{1}\geq t_{0}$,
and  the function $\pi(t)$ be defined by $(2.10)$. If $(2.11)$ holds, then there exists $t_{2}\geq t_{1}$ such that
\begin{equation}\label{2.15}
y(t)\leq k\pi(t)
\end{equation}
for every constant $k>0$ and $t\geq t_{2}$.
\end{corollary}

\begin{lemma}\label{2.3}
\ Assume $(H_{1})-(H_{4})$. Let $y\in S_{1}$  in $Lemma$ $ 2.1$ on $\big[t_{1}, \infty\big)$ for some $t_{1}\geq t_{0}$.
If $(2.11)$ holds, then there exists $t_{2}\geq t_{1}$ such that
\begin{equation}\label{2.16}
\frac{y(t)}{\pi_{1} (t)}\uparrow 
\end{equation}
for $t\geq t_{2}$.
\end{lemma}

\begin{proof}
\ Let $y\in S_{1}$  in $Lemma$ $ 2.1$ on $\big[t_{1}, \infty\big)$ for some $t_{1}\geq t_{0}$.
It follows from the monotonicity of $L_{1}y$ that, for $\ell\geq t$,
$$
\aligned
y(t)
&\geq -\int^{\ell}_{t}r_{1}^{-\frac{1}{\alpha}}(s)\left(L_{1}y(s)\right)^{\frac{1}{\alpha}}\text{d}s\\
&\geq -\left(L_{1}y(t)\right)^{\frac{1}{\alpha}}\int^{\ell}_{t}r_{1}^{-\frac{1}{\alpha}}(s)\text{d}s.
\endaligned
$$
Letting $\ell$ to $\infty$, we have
\begin{equation}\label{2.17}
 y(t)\geq -\left(L_{1}y(t)\right)^{\frac{1}{\alpha}}\cdot\pi_{1} (t).
\end{equation}
From $(2.17)$, we conclude that $y/\pi_{1}$ is nondecreasing, since
\begin{equation}\label{2.18}
 \left(\frac{y}{\pi_{1}}\right)'(t)=\frac{\left(L_{1}y(t)\right)^{\frac{1}{\alpha}}\pi_{1}(t)+y(t)}
{r_{1}^{\frac{1}{\alpha}}(t)\pi_{1}^{2}(t)}\geq 0.
 \end{equation}
 The proof is complete.
\end{proof}

\begin{theorem}\label{thm2.2}
\ Assume $(H_{1})-(H_{4})$. If
\begin{equation}\label{2.19}
\int_{t_{0}}^{\infty}r_{1}^{-\frac{1}{\alpha}}(v)\left(\int_{t_{0}}^{v}r_{2}^{-\frac{1}{\beta}}(u)
\left(\int_{t_{0}}^{u}\pi^{\gamma}\left(\sigma(s)\right)q(s)\text{d}s\right)^{\frac{1}{\beta}}\text{d}u\right)^{\frac{1}{\alpha}}\text{d}v = \infty ,
\end{equation}
then $(1.1)$ has property $A$.
\end{theorem}

\begin{proof}
\ Suppose on the contrary and assume that $y$ is a nonoscillatory solution of $(1.1)$ on $\big[t_{0},\infty\big)$.
Without loss of generality, we may assume  that $y(t)>0$ and $ y\left(\sigma(t)\right)>0$
for $t\in \big[t_{1},\infty\big)\subseteqq\big[t_{0},\infty\big)$. Then we obtain that $y$ eventually belongs to
one  of the four classes in $Lemma$ $ 2.1$.
We will consider each of them separately.

Assume $y\in S_{1}$. Note that $(2.3)$ and $(2.11)$ are necessary for $(2.19)$ to be valid.
In fact, since the function $\int_{t_{0}}^{t}\pi^{\gamma}\left(\sigma(s)\right)q(s)\text{d}s$
is unbounded due to $(H_{2})$ and $\pi'<0$, $(2.3)$ and $(2.11)$ must be hold.
Furthermore, by $(2.19)$, we see that $(2.1)$ holds, and we also obtain
\begin{equation}\label{2.20}
\int_{t_{0}}^{\infty}r_{1}^{-\frac{1}{\alpha}}(v)\left(\int_{t_{0}}^{v}r_{2}^{-\frac{1}{\beta}}(u)
\left(\int_{t_{0}}^{u}\pi_{1}^{\gamma}\left(\sigma(s)\right)q(s)\text{d}s\right)^{\frac{1}{\beta}}\text{d}u\right)^{\frac{1}{\alpha}}\text{d}v = \infty.
\end{equation}
Then by $Lemma$ $ 2.3$, it follows from $(2.16)$ that  there exist  $c> 0$ and $t_{2}\geq t_{1}$
such that $y(t)\geq c\pi_{1}(t)$ for $t\geq t_{2}$. Substituting this inequality into $(1.1)$,
we obtain
\begin{equation}\label{2.21}
-\left(L_{2}y\right)'(t)=q(t)y^{\gamma}\left(\sigma(t)\right)\geq c^{\gamma} q(t)\pi_{1}^{\gamma}\left(\sigma(t)\right),
\ \ \left(t\geq t_{2}\right).
\end{equation}
Integrating $(2.21)$  from $t_{2}$ to $t$, we have
$$\aligned -L_{2}y(t)
&\geq -L_{2}y\left(t_{2}\right)+c^{\gamma}\int_{t_{2}}^{t}q(s)\pi_{1}^{\gamma}\left(\sigma(s)\right)\text{d}s\\
&\geq c^{\gamma}\int_{t_{2}}^{t}\pi_{1}^{\gamma}\left(\sigma(s)\right)q(s)\text{d}s,
\endaligned$$
that is,
$$
-\left(L_{1}y\right)'(t)\geq c^{\frac{\gamma}{\beta}}r_{2}^{-\frac{1}{\beta}}(t)
\left(\int_{t_{2}}^{t}\pi_{1}^{\gamma}\left(\sigma(s)\right)q(s)\text{d}s\right)^{\frac{1}{\beta}}.
$$
Integrating the above inequality   from $t_{2}$ to $t$, we have
$$\aligned -L_{1}y(t)
&\geq -L_{1}y\left(t_{2}\right)+c^{\frac{\gamma}{\beta}}\int_{t_{2}}^{t}r_{2}^{-\frac{1}{\beta}}(u)
\left(\int_{t_{2}}^{u}\pi_{1}^{\gamma}\left(\sigma(s)\right)q(s)\text{d}s\right)^{\frac{1}{\beta}}\text{d}u\\
&\geq c^{\frac{\gamma}{\beta}}\int_{t_{2}}^{t}r_{2}^{-\frac{1}{\beta}}(u)
\left(\int_{t_{2}}^{u}\pi_{1}^{\gamma}\left(\sigma(s)\right)q(s)\text{d}s\right)^{\frac{1}{\beta}}\text{d}u,
\endaligned$$
that is,
\begin{equation}\label{2.22}
-y'(t)\geq c^{\frac{\gamma}{\alpha\beta}}r_{1}^{-\frac{1}{\alpha}}(t)\left(\int_{t_{2}}^{t}r_{2}^{-\frac{1}{\beta}}(u)
\left(\int_{t_{2}}^{u}\pi_{1}^{\gamma}\left(\sigma(s)\right)
q(s)\text{d}s\right)^{\frac{1}{\beta}}\text{d}u\right)^{\frac{1}{\alpha}}.
\end{equation}
Integrating $(2.22)$  from $t_{2}$ to $t$ for the last time, and taking $(2.20)$ into account, we have
$$\aligned
y(t)
&\leq y\left(t_{2}\right)-c^{\frac{\gamma}{\alpha\beta}}
\int_{t_{2}}^{t}r_{1}^{-\frac{1}{\alpha}}(v)\left(\int_{t_{2}}^{v}r_{2}^{-\frac{1}{\beta}}(u)
\left(\int_{t_{2}}^{u}\pi_{1}^{\gamma}\left(\sigma(s)\right)q(s)
\text{d}s\right)^{\frac{1}{\beta}}\text{d}u\right)^{\frac{1}{\alpha}}\text{d}v\\
&\rightarrow -\infty ,\ \ \left(t\rightarrow\infty\right),
\endaligned$$
which contradicts to the positivity of $y$. Thus,  $S_{1}={\O}$.

Assume $y\in S_{2}$. Noting  $(2.1)$ is necessary for the validity of  $(2.20)$, we have $lim_{t\rightarrow\infty} y(t)=0$.

Finally, by noting  $(2.3)$ and $(2.2)$ are necessary for the validity of  $(2.19)$,
it follows immediately from $Remark$ $2.1$ that $ S_{3}=S_{4}={\O}$. The proof is complete.
\end{proof}

\begin{theorem}\label{thm2.3}
\ Assume $(H_{1})-(H_{4})$. If
\begin{equation}\label{2.23}
\limsup\limits_{t\rightarrow\infty}\pi_{1}^{\frac{\gamma}{\beta}}\left(\sigma(t)\right)
\int_{t_{1}}^{t}r_{2}^{-\frac{1}{\beta}}(u)
\left(\int_{t_{1}}^{u}q(s)\text{d}s\right)^{\frac{1}{\beta}}\text{d}u>1,
\end{equation}
for any $t_{1}\geq t_{0}$, and $\gamma=\alpha\beta$, then $(1.1)$ has property $A$.
\end{theorem}

\begin{proof}
\ On the contrary, suppose  that $y$ is a nonoscillatory solution of $(1.1)$ on $\big[t_{0},\infty\big)$.
Without loss of generality, we may assume  that $y(t)>0$ and $ y\left(\sigma(t)\right)>0$
for $t\in \big[t_{1},\infty\big)\subseteqq\big[t_{0},\infty\big)$. Then we obtain that $y$ eventually belongs to
one  of the four classes in $Lemma$ $ 2.1$.
We will consider each of them separately.

First, note that $(2.23)$ along with $\left(H_{2}\right)$ implies $(2.3)$ and $(2.2)$. Then, by $Theorem$ $2.1$,
we get $ S_{3}=S_{4}={\O}$. Moreover, if $y\in S_{2}$, then $lim_{t\rightarrow\infty} y(t)=0$.

Next, we consider the class $S_{1}$. Assume $y\in S_{1}$. Integrating $(1.1)$  from $t_{1}$ to $t$
and using the decreasing of $y$, we have
\begin{equation}\label{2.24}
\aligned -L_{2}y(t)
&= -L_{2}y\left(t_{1}\right)+\int_{t_{1}}^{t}q(s)y^{\gamma}\left(\sigma(s)\right)\text{d}s
\geq \int_{t_{1}}^{t}q(s)y^{\gamma}\left(\sigma(s)\right)\text{d}s\\
&\geq y^{\gamma}\left(\sigma(t)\right)\int_{t_{1}}^{t}q(s)\text{d}s,
\endaligned
\end{equation}
that is,
\begin{equation}\label{2.25}
-\left(L_{1}y\right)'(t)\geq y^{\frac{\gamma}{\beta}}\left(\sigma(t)\right)r_{2}^{-\frac{1}{\beta}}(t)
\left(\int_{t_{1}}^{t}q(s)\text{d}s\right)^{\frac{1}{\beta}}.
\end{equation}
Integrating the above inequality   from $t_{1}$ to $t$, we have
\begin{equation}\label{2.26}
\aligned -L_{1}y(t)
&\geq -L_{1}y\left(t_{1}\right)+\int_{t_{1}}^{t}y^{\frac{\gamma}{\beta}}\left(\sigma(u)\right)r_{2}^{-\frac{1}{\beta}}(u)
\left(\int_{t_{1}}^{u}q(s)\text{d}s\right)^{\frac{1}{\beta}}\text{d}u\\
&\geq y^{\frac{\gamma}{\beta}}\left(\sigma(t)\right)\int_{t_{1}}^{t}r_{2}^{-\frac{1}{\beta}}(u)
\left(\int_{t_{1}}^{u}q(s)\text{d}s\right)^{\frac{1}{\beta}}\text{d}u.
\endaligned
\end{equation}
As in the proof of $Lemma$ $2.3$, we obtain $(2.17)$, which along with $(2.26)$ leads to
$$
\aligned -L_{1}y(t)
&\geq -\left(L_{1}y\right)^{\frac{\gamma}{\alpha\beta}}\left(\sigma(t)\right)
\pi_{1}^{\frac{\gamma}{\beta}}\left(\sigma(t)\right)\int_{t_{1}}^{t}r_{2}^{-\frac{1}{\beta}}(u)
\left(\int_{t_{1}}^{u}q(s)\text{d}s\right)^{\frac{1}{\beta}}\text{d}u\\
&\geq -\left(L_{1}y\right)^{\frac{\gamma}{\alpha\beta}}(t)
\pi_{1}^{\frac{\gamma}{\beta}}\left(\sigma(t)\right)\int_{t_{1}}^{t}r_{2}^{-\frac{1}{\beta}}(u)
\left(\int_{t_{1}}^{u}q(s)\text{d}s\right)^{\frac{1}{\beta}}\text{d}u.
\endaligned
$$
Taking $\gamma=\alpha\beta$ into account, the above inequality becomes
$$-L_{1}y(t)\geq -L_{1}y(t)\pi_{1}^{\frac{\gamma}{\beta}}\left(\sigma(t)\right)\int_{t_{1}}^{t}r_{2}^{-\frac{1}{\beta}}(u)
\left(\int_{t_{1}}^{u}q(s)\text{d}s\right)^{\frac{1}{\beta}}\text{d}u,$$
which results in a contradiction
$$\limsup\limits_{t\rightarrow\infty}\pi_{1}^{\frac{\gamma}{\beta}}\left(\sigma(t)\right)
\int_{t_{1}}^{t}r_{2}^{-\frac{1}{\beta}}(u)
\left(\int_{t_{1}}^{u}q(s)\text{d}s\right)^{\frac{1}{\beta}}\text{d}u\leq 1.$$
Thus,  $S_{1}={\O}$. The proof is complete.
\end{proof}

\begin{theorem}\label{thm2.4}
\ Assume $(H_{1})-(H_{4})$ and suppose that $(2.1)$ holds. If
\begin{equation}\label{2.27}
\limsup\limits_{t\rightarrow\infty}\pi_{1}^{\frac{\gamma}{\beta}}\left(\sigma(t)\right)
\int_{t_{0}}^{t}r_{2}^{-\frac{1}{\beta}}(u)
\left(\int_{t_{0}}^{u}q(s)\text{d}s\right)^{\frac{1}{\beta}}\text{d}u>1,
\end{equation}
 and $\gamma=\alpha\beta$, then $(1.1)$ has property $A$.
\end{theorem}

\begin{proof}
\ Using $Theorem$ $2.1$, we have $ S_{3}=S_{4}={\O}$, and if $y\in S_{2}$,
then $lim_{t\rightarrow\infty} y(t)=0$.

Now, we only need to consider the class $S_{1}$. Assume $y\in S_{1}$. As in the proof of $Theorem$ $2.3$,
we arrive at
$$
\aligned -L_{2}y(t)
&\geq -L_{2}y\left(t_{1}\right)+y^{\gamma}\left(\sigma(t)\right)\int_{t_{1}}^{t}q(s)\text{d}s\\
&\geq -L_{2}y\left(t_{1}\right)-y^{\gamma}\left(\sigma(t)\right)\int_{t_{0}}^{t_{1}}q(s)\text{d}s
+y^{\gamma}\left(\sigma(t)\right)\int_{t_{0}}^{t}q(s)\text{d}s,
\endaligned
$$
Since $lim_{t\rightarrow\infty} y(t)=0$, there exist $t_{2}>t_{1}$ such that
$$-L_{2}y\left(t_{1}\right)-y^{\gamma}\left(\sigma(t)\right)\int_{t_{0}}^{t_{1}}q(s)\text{d}s>0$$
for $t\geq t_{2}$. Thus, for $t\geq t_{2}$, we have
$$-L_{2}y(t)\geq y^{\gamma}\left(\sigma(t)\right)\int_{t_{0}}^{t}q(s)\text{d}s.$$
Integrating the above inequality   from $t_{2}$ to $t$, we have
$$
\aligned -L_{1}y(t)\geq
&-L_{1}y\left(t_{2}\right)-y^{\frac{\gamma}{\beta}}\left(\sigma(t)\right)\int_{t_{0}}^{t_{2}}r_{2}^{-\frac{1}{\beta}}(u)
\left(\int_{t_{0}}^{u}q(s)\text{d}s\right)^{\frac{1}{\beta}}\text{d}u\\
&+ y^{\frac{\gamma}{\beta}}\left(\sigma(t)\right)\int_{t_{0}}^{t}r_{2}^{-\frac{1}{\beta}}(u)
\left(\int_{t_{0}}^{u}q(s)\text{d}s\right)^{\frac{1}{\beta}}\text{d}u.
\endaligned
$$
There also exist $t_{3}>t_{2}$ such that
$$-L_{1}y\left(t_{2}\right)-y^{\frac{\gamma}{\beta}}\left(\sigma(t)\right)\int_{t_{0}}^{t_{2}}r_{2}^{-\frac{1}{\beta}}(u)
\left(\int_{t_{0}}^{u}q(s)\text{d}s\right)^{\frac{1}{\beta}}\text{d}u>0$$
for $t\geq t_{3}$. Thus, for $t\geq t_{3}$, we obtain
$$-L_{1}y(t)\geq
y^{\frac{\gamma}{\beta}}\left(\sigma(t)\right)\int_{t_{0}}^{t}r_{2}^{-\frac{1}{\beta}}(u)
\left(\int_{t_{0}}^{u}q(s)\text{d}s\right)^{\frac{1}{\beta}}\text{d}u.$$
The rest of proof is the same and hence we omit it. Finally, we obtain  $S_{1}={\O}$. The proof is complete.
\end{proof}

As following, we will establish various oscillation criteria for $(1.1)$.

\begin{theorem}\label{thm2.5}
\ Assume $(H_{1})-(H_{4})$. If
\begin{equation}\label{2.28}
\liminf\limits_{t\rightarrow\infty}\int_{t}^{\sigma(t)}r_{1}^{-\frac{1}{\alpha}}(v)
\left(\int_{t_{0}}^{v}r_{2}^{-\frac{1}{\beta}}(u)
\left(\int_{t_{0}}^{u}q(s)\text{d}s\right)^{\frac{1}{\beta}}\text{d}u\right)^{\frac{1}{\alpha}}\text{d}v
>\frac{1}{\text{e}} ,
\end{equation}
and
\begin{equation}\label{2.29}
\liminf\limits_{t\rightarrow\infty}\int_{t}^{\sigma\left(\sigma(t)\right)}r_{1}^{-\frac{1}{\alpha}}(v)
\left(\int_{v}^{\sigma(t)}r_{2}^{-\frac{1}{\beta}}(u)
\left(\int_{u}^{\sigma(t)}q(s)\text{d}s\right)^{\frac{1}{\beta}}\text{d}u\right)^{\frac{1}{\alpha}}\text{d}v
 >\frac{1}{\text{e}}
\end{equation}
hold, and moreover, $\alpha\beta=\gamma$,
then $(1.1)$ is oscillatory.
\end{theorem}

\begin{proof}
\ Suppose that $y$ is a nonoscillatory solution of $(1.1)$ on $\big[t_{0},\infty\big)$.
Without loss of generality, we may assume that $t_{1}\geq t_{0}$ such that $y(t)>0$ and $ y\left(\sigma(t)\right)>0$
for $t\geq t_{1}$. Then we obtain that $y$ eventually belongs to
one  of the four classes in $Lemma$ $ 2.1$. As following, we  consider each of these classes separately.

Assume $y\in S_{1}$. As in the proof of $Theorem$ $2.3$, we arrive at $(2.26)$, that is
$$y'+\left(r_{1}^{-\frac{1}{\alpha}}(t)\left(\int_{t_{1}}^{t}r_{2}^{-\frac{1}{\beta}}(u)
\left(\int_{t_{1}}^{u}q(s)\text{d}s\right)^{\frac{1}{\beta}}\text{d}u\right)^{\frac{1}{\alpha}}\right)
y^{\frac{\gamma}{\alpha\beta}}\left(\sigma(t)\right)\leq 0.$$
Using $\alpha\beta=\gamma$, the above inequality becomes
\begin{equation}\label{2.30}
y'+\left(r_{1}^{-\frac{1}{\alpha}}(t)\left(\int_{t_{1}}^{t}r_{2}^{-\frac{1}{\beta}}(u)
\left(\int_{t_{1}}^{u}q(s)\text{d}s\right)^{\frac{1}{\beta}}\text{d}u\right)^{\frac{1}{\alpha}}\right)
y\left(\sigma(t)\right)\leq 0.
\end{equation}
However, it is well-known (see, e.g., $[4,\  Theorem \ 2.4.1]$) that  condition $(2.28)$
implies oscillation of $(2.30)$. Thus, it is a contradiction with our initial assumption.
Then $S_{1}={\O}$.

Assume $y\in S_{2}$. Integrating $(1.1)$ from $t$ to $u$ $(t<u)$, and using the monotonicity of $y$, we obtain
$$L_{2}y(t)\geq L_{2}y(t)-L_{2}y(u)=\int_{t}^{u}q(s)y^{\gamma}\left(\sigma(s)\right)\text{d}s
\geq y^{\gamma}\left(\sigma(u)\right)\int_{t}^{u}q(s)\text{d}s,$$
that is,
$$\left(L_{1}y\right)'(t)
\geq y^{\frac{\gamma}{\beta}}\left(\sigma(u)\right)r_{2}^{-\frac{1}{\beta}}(t)
\left(\int_{t}^{u}q(s)\text{d}s\right)^{\frac{1}{\beta}}.$$
Integrating the above inequality from $t$ to $u$, we have
$$-L_{1}y(t)\geq y^{\frac{\gamma}{\beta}}\left(\sigma(u)\right)
\int_{t}^{u}r_{2}^{-\frac{1}{\beta}}(x)\left(\int_{x}^{u}q(s)\text{d}s\right)^{\frac{1}{\beta}}\text{d}x,$$
i.e.,
$$-y'(t)\geq y^{\frac{\gamma}{\alpha\beta}}\left(\sigma(u)\right)r_{1}^{-\frac{1}{\alpha}}(t)
\left(\int_{t}^{u}r_{2}^{-\frac{1}{\beta}}(x)\left(\int_{x}^{u}q(s)\text{d}s\right)^{\frac{1}{\beta}}\text{d}x
\right)^{\frac{1}{\alpha}}.$$
Taking $\gamma=\alpha\beta$ into account, we have
\begin{equation}\label{2.31}
-y'(t)\geq y\left(\sigma(u)\right)r_{1}^{-\frac{1}{\alpha}}(t)
\left(\int_{t}^{u}r_{2}^{-\frac{1}{\beta}}(x)\left(\int_{x}^{u}q(s)\text{d}s\right)^{\frac{1}{\beta}}\text{d}x
\right)^{\frac{1}{\alpha}}.
\end{equation}
Setting $u=\sigma(t)$ in $(2.31)$, we get
$$-y'(t)\geq y\left(\sigma\left(\sigma(t)\right)\right)r_{1}^{-\frac{1}{\alpha}}(t)
\left(\int_{t}^{\sigma(t)}r_{2}^{-\frac{1}{\beta}}(x)\left(\int_{x}^{\sigma(t)}q(s)
\text{d}s\right)^{\frac{1}{\beta}}\text{d}x
\right)^{\frac{1}{\alpha}},$$
i.e.,
\begin{equation}\label{2.32}
y'(t)+y\left(\sigma\left(\sigma(t)\right)\right)r_{1}^{-\frac{1}{\alpha}}(t)
\left(\int_{t}^{\sigma(t)}r_{2}^{-\frac{1}{\beta}}(x)\left(\int_{x}^{\sigma(t)}q(s)
\text{d}s\right)^{\frac{1}{\beta}}\text{d}x
\right)^{\frac{1}{\alpha}}\leq 0.
\end{equation}
However, condition $(2.29)$  implies oscillation of $(2.32)$, (see, e.g., $[4,\  Theorem \ 2.4.1]$).  It
means that $(1.1)$ cannot have a positive solution $y$ in the class $S_{2}$, which is a contradiction.
Thus, $S_{2}={\O}$.

Finally,  noting that $(2.1)$ is necessary for the validity of $(2.28)$, it follows immediately from $Remark$ $ 2.1$ that
$S_{3}=S_{4}={\O}$.
The proof is complete.
\end{proof}

The following results are simple consequences of those $Theorem$  mentioned above and $Corollary$ $2.1$.

\begin{theorem}\label{thm2.6}
\ Assume $(H_{1})-(H_{4})$. If $\gamma=\alpha\beta$, $(2.11)$ and $(2.28)$ hold, then all positive solutions
of $(1.1)$ satisfy $(2.15)$ for any $k>0$ and $t$  large enough.
\end{theorem}

\begin{theorem}\label{thm2.7}
\ Assume $(H_{1})-(H_{4})$. If $\gamma=\alpha\beta$, $(2.19)$ and $(2.29)$ hold,  then $(1.1)$ is oscillatory.
\end{theorem}

\begin{remark}
\ If
\begin{equation}\label{2.33}
\liminf\limits_{t\rightarrow\infty}\int_{t}^{\sigma(t)}r_{1}^{-\frac{1}{\alpha}}(v)
\left(\int_{v}^{\sigma(t)}r_{2}^{-\frac{1}{\beta}}(u)
\left(\int_{u}^{\sigma(t)}q(s)\text{d}s\right)^{\frac{1}{\beta}}\text{d}u\right)^{\frac{1}{\alpha}}\text{d}v
 >\frac{1}{\text{e}},
\end{equation}
holds, we have the validity of $(2.29)$. Thus, the conclusions of $Theorem$ $2.5$ and $2.7$ remain
valid if the condition $(2.29)$ is replaced by $(2.33)$.
\end{remark}

\begin{theorem}\label{thm2.8}
\ Assume $(H_{1})-(H_{4})$. If $\gamma=\alpha\beta$, $(2.23)$ and $(2.33)$ hold,  then $(1.1)$ is oscillatory.
\end{theorem}

\begin{theorem}\label{thm2.9}
\ Assume $(H_{1})-(H_{4})$. If $\gamma=\alpha\beta$, $(2.1)$, $(2.27)$ and $(2.33)$ hold,  then $(1.1)$ is oscillatory.
\end{theorem}

In order to prove the following conclusions, we recall an auxiliary result which is taken
from Wu et al. $[5,\ Lemma\ 2.3]$.

\begin{lemma}\label{2.4}$([5,\ Lemma\ 2.3])$
\ Let $g(u)=Au-B\left(u-C\right)^{\frac{\alpha +1}{\alpha}}$, where $B>0$, $A$ and $C$ are constants,
and $\alpha$ is a quotient of odd positive numbers. Then $g$ attains its maximum value on $ \mathbb{R}$
at $u^{*}=C+\left(\frac{\alpha A}{(\alpha+1)B}\right)^{\alpha}$ and
\begin{equation}\label{2.34}
\max_{u\in \mathbb{R}}g(u)=g\left(u^{*}\right)=AC+\frac{\alpha^{\alpha}}{(\alpha +1)^{\alpha +1}}
\cdot\frac{A^{\alpha +1}}{B^{\alpha}}.
\end{equation}
for $t\geq t_{2}$.
\end{lemma}

\begin{theorem}\label{thm2.10}
\ Assume $(H_{1})-(H_{4})$ and $\gamma=\alpha\beta$. If $(2.3)$ and $(2.33)$ hold, and also there
exists a function  $\rho\in C^{1}\left([t_{0},\infty),(0,\infty)\right)$ such that
\begin{equation}\label{2.35}
\limsup\limits_{t\rightarrow\infty}\left\{\frac{\pi_{1}^{\alpha}(t)}{\rho (t)}
\int_{T}^{t}\left(\rho(u)r_{2}^{-\frac{1}{\beta}}(u)
\left(\int_{T}^{u}q(s)\text{d}s\right)^{\frac{1}{\beta}}\left(\frac{\pi_{1}\left(\sigma(u)\right)}{\pi_{1}(u)}\right)
^{\alpha}-\frac{r_{1}(u)\left(\rho'(u)\right)^{\alpha +1}}{(\alpha +1)^{\alpha +1}\rho^{\alpha}(u)}
\right)\text{d}u\right\}>1,
\end{equation}
for any $T\in [t_{0},\infty)$,  then $(1.1)$ is oscillatory.
\end{theorem}

\begin{proof}
\ On the contrary, suppose  that $y$ is a nonoscillatory solution of $(1.1)$ on $\big[t_{0},\infty\big)$.
Without loss of generality, we may assume  that $y(t)>0$ and $ y\left(\sigma(t)\right)>0$
for $t\in \big[t_{1},\infty\big)\subseteqq\big[t_{0},\infty\big)$. Then we know that $y$ eventually belongs to
one  of the four classes in $Lemma$ $ 2.1$.
We will consider each of them separately.

Assume $y\in S_{1}$. Let's define the generalized Riccati Substitution
\begin{equation}\label{2.36}
w:=\rho\left(\frac{L_{1}y}{y^{\frac{\gamma}{\beta}}}+\frac{1}{\pi_{1}^{\frac{\gamma}{\beta}}}\right)
=\rho\left(\frac{L_{1}y}{y^{\alpha}}+\frac{1}{\pi_{1}^{\alpha}}\right)
\ \ \text{on}\ \big[t_{1},\infty\big).
\end{equation}
Taking $(2.17)$ into account, we see that $w\geq 0$ on $\big[t_{1},\infty\big)$.
Differentiating $(2.36)$, we arrive at
\begin{equation}\label{2.37}
\aligned w'
&=\frac{\rho'}{\rho}w+\rho\frac{\left(L_{1}y\right)'}{y^{\alpha}}
-\alpha\rho\frac{\left(L_{1}y\right)\cdot y'}{y^{\alpha +1}}
+\rho(-\alpha)\frac{-1}{\pi_{1}^{\alpha +1}\cdot r_{1}^{\frac{1}{\alpha}}}\\
&=\frac{\rho'}{\rho}w+\rho\frac{\left(L_{1}y\right)'}{y^{\alpha}}
-\frac{\alpha\rho}{r_{1}^{\frac{1}{\alpha}}}\left(\frac{L_{1}y}{y^{\alpha }}\right)^{\frac{\alpha +1}{\alpha}}
+\frac{\alpha\rho}{r_{1}^{\frac{1}{\alpha}}\pi_{1}^{\alpha +1}}\\
&=\frac{\rho'}{\rho}w+\rho\frac{\left(L_{1}y\right)'}{y^{\alpha}}
-\frac{\alpha}{\left(r_{1}\rho\right)^{\frac{1}{\alpha}}}
\left(w-\frac{\rho}{\pi_{1}^{\alpha }}\right)^{\frac{\alpha +1}{\alpha}}
+\frac{\alpha\rho}{r_{1}^{\frac{1}{\alpha}}\pi_{1}^{\alpha +1}}.
\endaligned
\end{equation}
As  the proof in $Theorem$ $2.3$, we arrive at  $(2.25)$.  Using $(2.16)$ in $(2.25)$,
we deduce that the inequality
\begin{equation}\label{2.38}
\aligned \left(L_{1}y\right)'(t)
&\leq -y^{\frac{\gamma}{\beta}}\left(\sigma(t)\right)r_{2}^{-\frac{1}{\beta}}(t)
\left(\int_{t_{2}}^{t}q(s)\text{d}s\right)^{\frac{1}{\beta}}\\
&\leq -y^{\alpha}\left(\sigma(t)\right)r_{2}^{-\frac{1}{\beta}}(t)
\left(\int_{t_{2}}^{t}q(s)\text{d}s\right)^{\frac{1}{\beta}}\\
&\leq -\left(\frac{\pi_{1}\left(\sigma(t)\right)}{\pi_{1}(t)}\right)^{\alpha}r_{2}^{-\frac{1}{\beta}}(t)
\left(\int_{t_{2}}^{t}q(s)\text{d}s\right)^{\frac{1}{\beta}}y^{\alpha}(t)
\endaligned
\end{equation}
holds for $t\geq t_{2}$, where $t_{2}\in \big[t_{1},\infty\big)$ is large enough. From
$(2.37)$ and $(2.38)$, it follows that
$$
\aligned
w'(t)\leq
&-\rho(t)\frac{\pi_{1}^{\alpha}\left(\sigma(t)\right)}{\pi_{1}^{\alpha}(t)}r_{2}^{-\frac{1}{\beta}}(t)
\left(\int_{t_{2}}^{t}q(s)\text{d}s\right)^{\frac{1}{\beta}}
+\frac{\rho'(t)}{\rho(t)}w(t)\\
&-\frac{\alpha}{\left(r_{1}(t)\rho(t)\right)^{\frac{1}{\alpha}}}
\left(w(t)-\frac{\rho(t)}{\pi_{1}^{\alpha }(t)}\right)^{\frac{\alpha +1}{\alpha}}
+\frac{\alpha\rho(t)}{r_{1}^{\frac{1}{\alpha}}(t)\pi_{1}^{\alpha +1}(t)}.
\endaligned
$$
Let
$$A:=\frac{\rho'(t)}{\rho(t)},\ \ B:=\frac{\alpha}{\left(r_{1}(t)\rho(t)\right)^{\frac{1}{\alpha}}},
\ \ C:=\frac{\rho(t)}{\pi_{1}^{\alpha }(t)}.$$
Using $(2.34)$ with the above inequality, we have
\begin{equation}\label{2.39}
\aligned
w'(t)
&\leq-\rho(t)r_{2}^{-\frac{1}{\beta}}(t)\left(\int_{t_{2}}^{t}q(s)\text{d}s\right)^{\frac{1}{\beta}}
\frac{\pi_{1}^{\alpha}\left(\sigma(t)\right)}{\pi_{1}^{\alpha}(t)}
+\frac{\rho'(t)}{\pi_{1}^{\alpha }(t)}
+\frac{r_{1}(t)\left(\rho'(t)\right)^{\alpha +1}}{(\alpha +1)^{\alpha +1}\rho^{\alpha}(t)}
+\frac{\alpha\rho(t)}{r_{1}^{\frac{1}{\alpha}}(t)\pi_{1}^{\alpha +1}(t)}\\
&=-\rho(t)r_{2}^{-\frac{1}{\beta}}(t)\left(\int_{t_{2}}^{t}q(s)\text{d}s\right)^{\frac{1}{\beta}}
\frac{\pi_{1}^{\alpha}\left(\sigma(t)\right)}{\pi_{1}^{\alpha}(t)}
+\left(\frac{\rho}{\pi_{1}^{\alpha }}\right)'(t)
+\frac{r_{1}(t)\left(\rho'(t)\right)^{\alpha +1}}{(\alpha +1)^{\alpha +1}\rho^{\alpha}(t)}.
\endaligned
\end{equation}
Integrating $(2.39)$ from $t_{2}$ to $t$, we obtain
$$
\aligned
&\int_{t_{2}}^{t}\left(\rho(u)r_{2}^{-\frac{1}{\beta}}(u)
\left(\int_{t_{2}}^{u}q(s)\text{d}s\right)^{\frac{1}{\beta}}\left(\frac{\pi_{1}\left(\sigma(u)\right)}{\pi_{1}(u)}\right)
^{\alpha}-\frac{r_{1}(u)\left(\rho'(u)\right)^{\alpha +1}}{(\alpha +1)^{\alpha +1}\rho^{\alpha}(u)}
\right)\text{d}u\\
&-\frac{\rho(t)}{\pi_{1}(t)}+\frac{\rho\left(t_{2}\right)}{\pi_{1}\left(t_{2}\right)}
\leq w\left(t_{2}\right)-w(t).
\endaligned
$$
Taking the definition of $w$ into account, we get
\begin{equation}\label{2.40}
\aligned
&\int_{t_{2}}^{t}\left(\rho(u)r_{2}^{-\frac{1}{\beta}}(u)
\left(\int_{t_{2}}^{u}q(s)\text{d}s\right)^{\frac{1}{\beta}}\left(\frac{\pi_{1}\left(\sigma(u)\right)}{\pi_{1}(u)}\right)
^{\alpha}-\frac{r_{1}(u)\left(\rho'(u)\right)^{\alpha +1}}{(\alpha +1)^{\alpha +1}\rho^{\alpha}(u)}
\right)\text{d}u\\
&\leq \rho\left(t_{2}\right)\frac{L_{1}y\left(t_{2}\right)}{y^{\alpha}\left(t_{2}\right)}
-\rho(t)\frac{L_{1}y(t)}{y^{\alpha}(t)}.
\endaligned
\end{equation}
On the other hand, from $(2.17)$, it follows that
$$-\frac{\rho(t)}{\pi_{1}^{\alpha }(t)}\leq \rho(t)\frac{L_{1}y(t)}{y^{\alpha}(t)}\leq 0.$$
Substituting the above estimate into $(2.40)$, we get
\begin{equation}\label{2.41}
\aligned
&\int_{t_{2}}^{t}\left(\rho(u)r_{2}^{-\frac{1}{\beta}}(u)
\left(\int_{t_{2}}^{u}q(s)\text{d}s\right)^{\frac{1}{\beta}}\left(\frac{\pi_{1}\left(\sigma(u)\right)}{\pi_{1}(u)}\right)
^{\alpha}-\frac{r_{1}(u)\left(\rho'(u)\right)^{\alpha +1}}{(\alpha +1)^{\alpha +1}\rho^{\alpha}(u)}
\right)\text{d}u\\
&\leq \frac{\rho(t)}{\pi_{1}^{\alpha }(t)}.
\endaligned
\end{equation}
Multiplying $(2.41)$ by $\pi_{1}^{\alpha }(t)/\rho(t)$ and taking the $limsup$ on both
sides of the resulting inequality, we obtain a contradiction with $(2.35)$. Thus, $S_{1}={\O}$.

Assume $y\in S_{2}$. As in the proof of $Theorem$ $2.5$, one arrives at contradiction with $(2.33)$.
Thus, $S_{2}={\O}$.

As following, we show $S_{3}=S_{4}={\O}$. Since $(2.3)$ holds due to $(H_{2})$, then the function
$$\int_{t_{0}}^{t}r_{2}^{-\frac{1}{\beta}}(u)
\left(\int_{t_{0}}^{u}q(s)\text{d}s\right)^{\frac{1}{\beta}}\text{d}u$$
is unbounded, and so $(2.2)$ holds. The rest of proof proceeds in the same manner as that of  $Theorem$ $2.1$.
The proof is complete.
\end{proof}

Depending on the appropriate choice of the function $\rho$, we can use  $Theorem$ $2.10$ in a wide range
of applications for studying the oscillation of $(1.1)$. Thus, by choosing $\rho(t)=\pi_{1}^{\alpha }(t)$,
 $\rho(t)=\pi_{1}(t)$  and $\rho(t)=1$, we obtain the following results, respectively.

\begin{corollary}
\ Assume $(H_{1})-(H_{4})$ and $\gamma=\alpha\beta$. Moreover, assume that $(2.3)$ and $(2.33)$ hold. If
\begin{equation}\label{2.42}
\limsup\limits_{t\rightarrow\infty}
\int_{T}^{t}\left(r_{2}^{-\frac{1}{\beta}}(u)
\left(\int_{T}^{u}q(s)\text{d}s\right)^{\frac{1}{\beta}}\pi_{1}^{\alpha}\left(\sigma(u)\right)
-\left(\frac{\alpha}{\alpha +1}\right)^{\alpha +1}\frac{1}{r_{1}^{\frac{1}{\alpha}}(u)\pi_{1}(u)}
\right)\text{d}u>1,
\end{equation}
for any $T\in [t_{0},\infty)$,  then $(1.1)$ is oscillatory.
\end{corollary}

\begin{corollary}
\ Assume $(H_{1})-(H_{4})$ and $\gamma=\alpha\beta$. Moreover, assume that $(2.3)$ and $(2.33)$ hold. If
\begin{equation}\label{2.43}
\limsup\limits_{t\rightarrow\infty}\pi_{1}^{\alpha-1 }(t)
\int_{T}^{t}\left(r_{2}^{-\frac{1}{\beta}}(u)
\left(\int_{T}^{u}q(s)\text{d}s\right)^{\frac{1}{\beta}}
\frac{\pi_{1}^{\alpha}\left(\sigma(u)\right)}{\pi_{1}^{\alpha-1 }(u)}
-\frac{1}{\left(\alpha+1\right)^{\alpha+1}r_{1}^{\frac{1}{\alpha}}(u)\pi_{1}^{\alpha}(u)}
\right)\text{d}u>1,
\end{equation}
for any $T\in [t_{0},\infty)$,  then $(1.1)$ is oscillatory.
\end{corollary}

\begin{corollary}
\ Assume $(H_{1})-(H_{4})$ and $\gamma=\alpha\beta$. Moreover, assume that $(2.3)$ and $(2.33)$ hold. If
\begin{equation}\label{2.44}
\limsup\limits_{t\rightarrow\infty}\pi_{1}^{\alpha}(t)
\int_{T}^{t}r_{2}^{-\frac{1}{\beta}}(u)
\left(\int_{T}^{u}q(s)\text{d}s\right)^{\frac{1}{\beta}}
\left(\frac{\pi_{1}\left(\sigma(u)\right)}{\pi_{1}(u)}
\right)^{\alpha}\text{d}u>1,
\end{equation}
for any $T\in [t_{0},\infty)$,  then $(1.1)$ is oscillatory.
\end{corollary}

\begin{remark}
\  The conclusions of $Theorem$ $2.10$ and $Corollary$ $2.2-2.4$ remain
valid if the condition $(2.3)$ is replaced by $(2.1)$.
\end{remark}

\begin{lemma}\label{2.5}
\ Assume $(H_{1})-(H_{4})$ and $\gamma=\alpha\beta$. Furthermore, assume that $(2.1)$ holds.
Suppose that $(1.1)$ has a positive solution $y\in S_{1}$ on
$\big[t_{1}, \infty\big)\subseteqq\big[t_{0}, \infty\big)$ and that $\lambda$ and $\mu$ are
constants satisfying
\begin{equation}\label{2.45}
0\leq \lambda+\mu<\alpha,
\end{equation}
\begin{equation}\label{2.46}
0\leq\lambda\leq r_{2}^{-\frac{1}{\beta}}(t)
\left(\int_{t_{1}}^{t}q(s)\text{d}s\right)^{\frac{1}{\beta}}
\pi_{1}^{\alpha}\left(\sigma(t)\right)\pi_{1}(t)r_{1}^{\frac{1}{\alpha}}(t)
\end{equation}
and
\begin{equation}\label{2.47}
0\leq\mu\leq \alpha \left(\int_{t_{1}}^{t}r_{2}^{-\frac{1}{\beta}}(u)
\left(\int_{t_{1}}^{u}q(s)\text{d}s\right)^{\frac{1}{\beta}}\text{d}u\right)^{\frac{1}{\alpha}}
\pi_{1}\left(\sigma(t)\right).
\end{equation}
Then there exists a $t_{*}\in\big[t_{1}, \infty\big)$ such that
\begin{equation}\label{2.48}
\frac{y}{\pi_{1}^{1-\frac{\lambda}{\alpha}}}\uparrow
\end{equation}
and
\begin{equation}\label{2.49}
\frac{y}{\pi_{1}^{\frac{\mu}{\alpha}}}\downarrow
\end{equation}
on $\big[t_{*}, \infty\big)$.
\end{lemma}

\begin{proof}
\ Assume $y\in S_{1}$. As the proof in $Theorem$ $2.3$, we arrive at $(2.25)$.
From $(1.1)$, $(2.17)$ and $(2.37)$, we see that
$$
\aligned
\left(-\left(L_{1}y\right)\cdot\pi_{1}^{\lambda}\right)'(t)
&=-\left(L_{1}y\right)'(t)\pi_{1}^{\lambda}(t)
+L_{1}y(t)\lambda\pi_{1}^{\lambda-1}(t)r_{1}^{-\frac{1}{\alpha}}(t)\\
&\geq r_{2}^{-\frac{1}{\beta}}(t)
\left(\int_{t_{1}}^{t}q(s)\text{d}s\right)^{\frac{1}{\beta}}
y^{\frac{\gamma}{\beta}}\left(\sigma(t)\right)\pi_{1}^{\lambda}(t)
+\lambda L_{1}y(t)\pi_{1}^{\lambda-1}(t)r_{1}^{-\frac{1}{\alpha}}(t)\\
&=r_{2}^{-\frac{1}{\beta}}(t)
\left(\int_{t_{1}}^{t}q(s)\text{d}s\right)^{\frac{1}{\beta}}
y^{\alpha}\left(\sigma(t)\right)\pi_{1}^{\lambda}(t)
+\lambda L_{1}y(t)\pi_{1}^{\lambda-1}(t)r_{1}^{-\frac{1}{\alpha}}(t)\\
&\geq -r_{2}^{-\frac{1}{\beta}}(t)
\left(\int_{t_{1}}^{t}q(s)\text{d}s\right)^{\frac{1}{\beta}}
L_{1}y\left(\sigma(t)\right)\pi_{1}^{\alpha}\left(\sigma(t)\right)\pi_{1}^{\lambda}(t)\\
&\ \ \ \ +\lambda L_{1}y(t)\pi_{1}^{\lambda-1}(t)r_{1}^{-\frac{1}{\alpha}}(t)\\
&\geq -r_{2}^{-\frac{1}{\beta}}(t)
\left(\int_{t_{1}}^{t}q(s)\text{d}s\right)^{\frac{1}{\beta}}
L_{1}y(t)\pi_{1}^{\alpha}\left(\sigma(t)\right)\pi_{1}^{\lambda}(t)\\
&\ \ \ \ +\lambda L_{1}y(t)\pi_{1}^{\lambda-1}(t)r_{1}^{-\frac{1}{\alpha}}(t)\\
&=-L_{1}y(t)\pi_{1}^{\lambda}(t)
\left(r_{2}^{-\frac{1}{\beta}}(t)
\left(\int_{t_{1}}^{t}q(s)\text{d}s\right)^{\frac{1}{\beta}}\pi_{1}^{\alpha}\left(\sigma(t)\right)
-\frac{\lambda}{r_{1}^{\frac{1}{\alpha}}(t)\pi_{1}(t)}\right)\\
&\geq 0.
\endaligned
$$
Thus, $-\left(L_{1}y\right)\pi_{1}^{\lambda}$ is nondecreasing eventually, say for $y\geq t_{2}$,
where $t_{2}\in \big[t_{1}, \infty\big)$ is large enough. Furthermore, using this property, we
get
\begin{equation}\label{2.50}
\aligned
y(t)
&\geq -\int^{\infty}_{t}r_{1}^{-\frac{1}{\alpha}}(s)\left(L_{1}y\right)^{\frac{1}{\alpha}}(s)\text{d}s\\
&=-\int^{\infty}_{t}r_{1}^{-\frac{1}{\alpha}}(s)
\frac{\pi_{1}^{\frac{\lambda}{\alpha}}(s)}{\pi_{1}^{\frac{\lambda}{\alpha}}(s)}
\left(L_{1}y\right)^{\frac{1}{\alpha}}(s)\text{d}s\\
&\geq -\left(\left(L_{1}y\right)\cdot\pi_{1}^{\lambda}\right)^{\frac{1}{\alpha}}(t)
\int_{t}^{\infty}\frac{1}{r_{1}^{\frac{1}{\alpha}}(s)\pi_{1}^{\frac{\lambda}{\alpha}}(s)}\text{d}s.
\endaligned
\end{equation}
It is easy to verify that
\begin{equation}\label{2.51}
\int_{t}^{\infty}\frac{1}{r_{1}^{\frac{1}{\alpha}}(s)\pi_{1}^{\frac{\lambda}{\alpha}}(s)}\text{d}s
=\frac{\pi_{1}^{1-\frac{\lambda}{\alpha}}(t)}{1-\frac{\lambda}{\alpha}},
\end{equation}
and thus, we get
\begin{equation}\label{2.52}
y(t)\geq -\frac{\left(L_{1}y\right)^{\frac{1}{\alpha}}(t)\cdot\pi_{1}(t)}{1-\frac{\lambda}{\alpha}}
=-\frac{r_{1}^{\frac{1}{\alpha}}(t)y'(t)\pi_{1}(t)}{1-\frac{\lambda}{\alpha}}.
\end{equation}
Therefore,
$$\left(\frac{y}{\pi_{1}^{1-\frac{\lambda}{\alpha}}}\right)'(t)
=\frac{r_{1}^{\frac{1}{\alpha}}(t)y'(t)\pi_{1}(t)+\left(1-\frac{\lambda}{\alpha}\right)y(t)}
{r_{1}^{\frac{1}{\alpha}}(t)\pi_{1}^{2-\frac{\lambda}{\alpha}}(t)}
\geq 0,$$
and thus, $y/\pi_{1}^{1-\frac{\lambda}{\alpha}}$ is nondecreasing.

Next, we will prove the last monotonicity. As the proof in $Theorem$ $2.3$,
we arrive at $(2.26)$, that is
$$-r_{1}(t)\left(y'(t)\right)^{\alpha}
\geq y^{\alpha}\left(\sigma(t)\right)\int_{t_{1}}^{t}r_{2}^{-\frac{1}{\beta}}(u)
\left(\int_{t_{1}}^{u}q(s)\text{d}s\right)^{\frac{1}{\beta}}\text{d}u.$$
Using $(2.16)$ with the above inequality, we have
$$-r_{1}(t)\left(y'(t)\right)^{\alpha}
\geq \frac{\pi_{1}^{\alpha}\left(\sigma(t)\right)}{\pi_{1}^{\alpha}(t)}
y^{\alpha}(t)\int_{t_{2}}^{t}r_{2}^{-\frac{1}{\beta}}(u)
\left(\int_{t_{2}}^{u}q(s)\text{d}s\right)^{\frac{1}{\beta}}\text{d}u,$$
i.e.,
$$y(t)\leq -r_{1}^{\frac{1}{\alpha}}(t)y'(t)
\frac{\pi_{1}(t)}{\pi_{1}\left(\sigma(t)\right)}
\left(\int_{t_{2}}^{t}r_{2}^{-\frac{1}{\beta}}(u)
\left(\int_{t_{2}}^{u}q(s)\text{d}s\right)^{\frac{1}{\beta}}\text{d}u\right)^{-\frac{1}{\alpha}}$$
for $t\geq t_{2}$, where $t_{2}\geq t_{1}$.
Using the above relation in the equality
$$\left(\frac{y}{\pi_{1}^{\frac{\mu}{\alpha}}}\right)'(t)
=\frac{y'(t)}{\pi_{1}^{\frac{\mu}{\alpha}}(t)}
+\frac{\frac{\mu}{\alpha}y(t)}{\pi_{1}^{\frac{\mu}{\alpha}+1}(t)r_{1}^{\frac{1}{\alpha}}(t)},$$
and taking the condition $(2.47)$ into account, we get
$$
\aligned
\left(\frac{y}{\pi_{1}^{\frac{\mu}{\alpha}}}\right)'(t)
&\leq \frac{y'(t)}{\pi_{1}^{\frac{\mu}{\alpha}}(t)}
-\frac{\frac{\mu}{\alpha}y'(t)}{\pi_{1}^{\frac{\mu}{\alpha}(t)}\pi_{1}\left(\sigma(t)\right)}
\left(\int_{t_{2}}^{t}r_{2}^{-\frac{1}{\beta}}(u)
\left(\int_{t_{2}}^{u}q(s)\text{d}s\right)^{\frac{1}{\beta}}\text{d}u\right)^{-\frac{1}{\alpha}}\\
&=\frac{y'(t)}{\pi_{1}^{\frac{\mu}{\alpha}}(t)}
\left(1-\frac{\mu}{\alpha\pi_{1}\left(\sigma(t)\right)}
\left(\int_{t_{2}}^{t}r_{2}^{-\frac{1}{\beta}}(u)
\left(\int_{t_{2}}^{u}q(s)\text{d}s\right)^{\frac{1}{\beta}}\text{d}u\right)^{-\frac{1}{\alpha}}\right)\\
&\leq \frac{y'(t)}{\pi_{1}^{\frac{\mu}{\alpha}}(t)}
\left(1-\frac{\mu}{\alpha\pi_{1}\left(\sigma(t)\right)}
\left(\int_{t_{1}}^{t}r_{2}^{-\frac{1}{\beta}}(u)
\left(\int_{t_{1}}^{u}q(s)\text{d}s\right)^{\frac{1}{\beta}}\text{d}u\right)^{-\frac{1}{\alpha}}\right)\\
&\leq 0.
\endaligned
$$
Thus, $y/\pi_{1}^{\frac{\mu}{\alpha}}$ is nonincreasing.
 The proof is complete.
\end{proof}

\begin{theorem}\label{thm2.11}
\ Assume $(H_{1})-(H_{4})$ and $\gamma=\alpha\beta$. Furthermore, suppose that  $(2.33)$ holds
and $\lambda$ and $\mu$ are constants satisfying $(2.45)-(2.47)$. If
\begin{equation}\label{2.53}
\limsup\limits_{t\rightarrow\infty}\pi_{1}^{\lambda}(t)
\pi_{1}^{\alpha-\lambda-\mu}\left(\sigma(t)\right)
\int_{t_{1}}^{t}\pi_{1}^{\mu}\left(\sigma(u)\right)r_{2}^{-\frac{1}{\beta}}(u)
\left(\int_{t_{1}}^{u}q(s)\text{d}s\right)^{\frac{1}{\beta}}\text{d}u
>\left(1-\frac{\lambda}{\alpha}\right)^{\alpha},
\end{equation}
for any $t_{1}\geq t_{0}$, then $(1.1)$ is oscillatory.
\end{theorem}

\begin{proof}
\ Suppose the contrary and assume  that $y$ is a nonoscillatory solution of $(1.1)$ on $\big[t_{0},\infty\big)$.
Without loss of generality, we may assume  that $y(t)>0$ and $ y\left(\sigma(t)\right)>0$
for $t\in \big[t_{1},\infty\big)\subseteqq\big[t_{0},\infty\big)$. Then we know that $y$ eventually belongs to
one  of the four classes in $Lemma$ $ 2.1$.
We will consider each of them separately.

Before proceeding further, note that $(2.11)$ and
\begin{equation}\label{2.54}
\int_{t_{0}}^{\infty}q(s)\pi_{1}^{\gamma}\left(\sigma(s)\right)\text{d}s=\infty
\end{equation}
are necessary for $(2.19)$ to be valid. To verify this, it suffices to see that
$(H_{2})$ implies
\begin{equation}\label{2.55}
\pi_{1}^{\frac{\lambda}{\alpha}}(t)\pi_{1}^{1-\frac{\lambda}{\alpha}-\frac{\mu}{\alpha}}\left(\sigma(t)\right)
\leq \pi_{1}^{1-\frac{\lambda}{\alpha}}(t)\rightarrow 0,\ \ \ (t\rightarrow\infty).
\end{equation}
From the above inequality, we conclude that the function $\int_{t_{0}}^{t}\pi_{1}^{\mu}\left(\sigma(u)\right)r_{2}^{-\frac{1}{\beta}}(u)
\left(\int_{t_{1}}^{u}q(s)\text{d}s\right)^{\frac{1}{\beta}}\text{d}u$
and consequently $\int_{t_{1}}^{t}r_{2}^{-\frac{1}{\beta}}(u)
\left(\int_{t_{1}}^{u}q(s)\text{d}s\right)^{\frac{1}{\beta}}\text{d}u$
must be unbounded.

Assume $y\in S_{1}$. As the proof in $Theorem$ $2.3$, we arrive at $(2.26)$, that is
\begin{equation}\label{2.56}
\aligned
-r_{1}(t)\left(y'(t)\right)^{\alpha}
&\geq -r_{1}(t_{1})\left(y'(t_{1})\right)^{\alpha}
+\int_{t_{1}}^{t}y^{\frac{\gamma}{\beta}}\left(\sigma(u)\right)r_{2}^{-\frac{1}{\beta}}(u)
\left(\int_{t_{1}}^{u}q(s)\text{d}s\right)^{\frac{1}{\beta}}\text{d}u\\
&\geq \int_{t_{1}}^{t}y^{\alpha}\left(\sigma(u)\right)r_{2}^{-\frac{1}{\beta}}(u)
\left(\int_{t_{1}}^{u}q(s)\text{d}s\right)^{\frac{1}{\beta}}\text{d}u.
\endaligned
\end{equation}
Using the conclusions of $Lemma$ $2.5$ that $y/\pi_{1}^{\frac{\mu}{\alpha}}$ is nonincreasing
and $y/\pi_{1}^{1-\frac{\lambda}{\alpha}}$ is nondecreasing, we obtain
\begin{equation}\label{2.57}
\aligned
-r_{1}(t)\left(y'(t)\right)^{\alpha}
&\geq \int_{t_{1}}^{t}\frac{y^{\alpha}\left(\sigma(u)\right)}{\pi_{1}^{\mu}\left(\sigma(u)\right)}
\pi_{1}^{\mu}\left(\sigma(u)\right)r_{2}^{-\frac{1}{\beta}}(u)
\left(\int_{t_{1}}^{u}q(s)\text{d}s\right)^{\frac{1}{\beta}}\text{d}u\\
&\geq \left(\frac{y\left(\sigma(t)\right)}{\pi_{1}^{\frac{\mu}{\alpha}}\left(\sigma(t)\right)}\right)^{\alpha}
\int_{t_{1}}^{t}\pi_{1}^{\mu}\left(\sigma(u)\right)r_{2}^{-\frac{1}{\beta}}(u)
\left(\int_{t_{1}}^{u}q(s)\text{d}s\right)^{\frac{1}{\beta}}\text{d}u\\
&= \left(\frac{y\left(\sigma(t)\right)\pi_{1}^{1-\frac{\lambda}{\alpha}}\left(\sigma(t)\right)}
{\pi_{1}^{\frac{\mu}{\alpha}}\left(\sigma(t)\right)\pi_{1}^{1-\frac{\lambda}{\alpha}}\left(\sigma(t)\right)}\right)^{\alpha}
\int_{t_{1}}^{t}\pi_{1}^{\mu}\left(\sigma(u)\right)r_{2}^{-\frac{1}{\beta}}(u)
\left(\int_{t_{1}}^{u}q(s)\text{d}s\right)^{\frac{1}{\beta}}\text{d}u\\
&\geq \left(\frac{y(t)\pi_{1}^{1-\frac{\lambda}{\alpha}-\frac{\mu}{\alpha}}\left(\sigma(t)\right)}
{\pi_{1}^{1-\frac{\lambda}{\alpha}}(t)}\right)^{\alpha}
\int_{t_{1}}^{t}\pi_{1}^{\mu}\left(\sigma(u)\right)r_{2}^{-\frac{1}{\beta}}(u)
\left(\int_{t_{1}}^{u}q(s)\text{d}s\right)^{\frac{1}{\beta}}\text{d}u.
\endaligned
\end{equation}
Using $(2.52)$ in the above inequality, we have
$$
\aligned
&-r_{1}(t)\left(y'(t)\right)^{\alpha}\\
&\geq -r_{1}(t)\left(y'(t)\right)^{\alpha}
\left(\frac{\pi_{1}^{\frac{\lambda}{\alpha}}(t)\pi_{1}^{1-\frac{\lambda}{\alpha}
-\frac{\mu}{\alpha}}\left(\sigma(t)\right)}
{1-\frac{\lambda}{\alpha}}\right)^{\alpha}
\int_{t_{1}}^{t}\pi_{1}^{\mu}\left(\sigma(u)\right)r_{2}^{-\frac{1}{\beta}}(u)
\left(\int_{t_{1}}^{u}q(s)\text{d}s\right)^{\frac{1}{\beta}}\text{d}u,
\endaligned
$$
that is,
$$1\geq \left(\frac{\pi_{1}^{\frac{\lambda}{\alpha}}(t)\pi_{1}^{1-\frac{\lambda}{\alpha}
-\frac{\mu}{\alpha}}\left(\sigma(t)\right)}
{1-\frac{\lambda}{\alpha}}\right)^{\alpha}
\int_{t_{1}}^{t}\pi_{1}^{\mu}\left(\sigma(u)\right)r_{2}^{-\frac{1}{\beta}}(u)
\left(\int_{t_{1}}^{u}q(s)\text{d}s\right)^{\frac{1}{\beta}}\text{d}u.$$
Taking the $limsup$ on both sides of the above inequality, we reach a contradiction
with $(2.53)$. Thus, $S_{1}={\O}$.

Accounting to $Remark$ $2.2$ with $(2.33)$, we have $S_{2}={\O}$. Also,
using $Theorem$ $2.1$, we arrive at $S_{3}=S_{4}={\O}$.
The proof is complete.
\end{proof}

\begin{theorem}\label{thm2.12}
\ Assume $(H_{1})-(H_{4})$ and $\gamma=\alpha\beta$. Furthermore, suppose that $(2.3)$ and $(2.33)$ hold,
and $\lambda\in [0,\alpha)$ is a  constant satisfying $(2.46)$. If there
exist a function  $\rho\in C^{1}\left([t_{0},\infty),(0,\infty)\right)$ and $T\in [t_{0},\infty)$ such that
\begin{equation}\label{2.58}
\aligned
\limsup\limits_{t\rightarrow\infty}
&\left\{\frac{\pi_{1}^{\alpha}(t)}{\rho (t)}
\int_{T}^{t}\left(\rho(u)r_{2}^{-\frac{1}{\beta}}(u)
\left(\int_{T}^{u}q(s)\text{d}s\right)^{\frac{1}{\beta}}\left(\frac{\pi_{1}\left(\sigma(u)\right)}{\pi_{1}(u)}\right)
^{\alpha-\lambda}-\frac{r_{1}(u)\left(\rho'(u)\right)^{\alpha +1}}{(\alpha +1)^{\alpha +1}\rho^{\alpha}(u)}
\right)\text{d}u\right\}\\
&>1,
\endaligned
\end{equation}
then $(1.1)$ is oscillatory.
\end{theorem}

\begin{proof}
\ For the proof of this $Theorem$, it suffices to use $(2.48)$ instead of $(2.16)$ in $(2.25)$
in the proof of  $Theorem$ $2.10$.
\end{proof}

\begin{corollary}
\ Assume $(H_{1})-(H_{4})$ and $\gamma=\alpha\beta$. Furthermore, suppose that $(2.3)$ and $(2.33)$ hold
and $\lambda\in [0,\alpha)$ is a  constant satisfying $(2.46)$. If
\begin{equation}\label{2.59}
\aligned
\limsup\limits_{t\rightarrow\infty}
&\int_{T}^{t}\left(r_{2}^{-\frac{1}{\beta}}(u)
\left(\int_{T}^{u}q(s)\text{d}s\right)^{\frac{1}{\beta}}\pi_{1}^{\alpha-\lambda}\left(\sigma(u)\right)
\pi_{1}^{\lambda}(u)
-\left(\frac{\alpha}{\alpha +1}\right)^{\alpha +1}\frac{1}{r_{1}^{\frac{1}{\alpha}}(u)\pi_{1}(u)}
\right)\text{d}u\\
&>1,
\endaligned
\end{equation}
for any $T\in [t_{0},\infty)$,  then $(1.1)$ is oscillatory.
\end{corollary}

\vspace{.4cm} \noindent{\bf 3. Examples}
 \setcounter{section}{3}
 \setcounter{equation}{0}

In this section, we illustrate the strength of our results using two
Euler-type differential equations, as two examples.

\begin{example}\label{3.1}
\ Consider the third-order advanced  differential equation
\begin{equation}\label{3.1}
\left(t^{3}\left(\left(t^{4}\left(y'(t)\right)^{\frac{5}{3}}\right)'\right)^{\frac{1}{7}}\right)'
+t^{6}y^{\frac{9}{5}}(2t)=0,\ \ \ t\geq 1.
\end{equation}
\end{example}
It is easy to verify that  the condition $(2.1)$ is satisfied. Using $Theorem\ 2.1$, we
obtain that $Eq.$ $(3.1)$ has $property$ $A$.

\begin{example}\label{3.2}
\ Consider the third-order advanced  differential equation
\begin{equation}\label{3.2}
\left(t^{n}\left(\left(t^{m}y'(t)\right)'\right)^{\frac{1}{3}}\right)'
+q_{0}t^{\frac{m}{3}+n-\frac{5}{3}}y^{\frac{1}{3}}\left(\delta t\right)=0,\ \  t\geq 1,
\end{equation}
where $m>1$, $n>\frac{1}{3}$, $q_{0}>0$ and $\delta\geq 1$.
\end{example}
Clearly, $r_{1}(t)=t^{m}$, $r_{2}(t)=t^{n}$, $\alpha=1$, $\beta=\frac{1}{3}$,
$\gamma=\alpha\beta=\frac{1}{3}$, $q(t)=q_{0}t^{\frac{m}{3}+n-\frac{5}{3}}$,
$\sigma(t)=\delta t$, and
$$
\pi_{1}(t)=\int_{t}^{\infty}r_{1}^{-\frac{1}{\alpha}}(s)\text{d}s
=\int_{t}^{\infty}s^{-m}\text{d}s
=\frac{t^{1-m}}{m-1}.
$$

$Theorem$ $2.1$ (on the asymptotic properties of nonoscillatory solutions).
It is easy to verify that the condition $(2.1)$ holds. Thus, any nonoscillatory,
say positive solution of $Eq.$ $(3.2)$ converges to zero as $t\rightarrow\infty$,
without any additional requirement.

As following, we consider the oscillation of $Eq.$ $(3.2)$.

After some computations, we note that the conditions $(2.23)$, $(2.28)$ and $(2.33)$
reduce to
\begin{equation}\label{3.3}
27q_{0}^{3}\delta^{1-m}>(m+3n-2)^{3}(m-1)^{2},
\end{equation}
\begin{equation}\label{3.4}
27q_{0}^{3}\ln\delta>\frac{(m+3n-2)^{3}(m-1)}{\text{e}},
\end{equation}
and
\begin{equation}\label{3.5}
\aligned
27q_{0}^{3}\big\{
&\frac{\delta^{m+3n-2}-1}{(m+3n-2)(3n-1)}
+\frac{\ln\delta}{m-1}\\
&+\frac{27\left(\delta^{\frac{2m+6n-4}{3}}-1\right)}{(m-6n+1)(2m+6n-4)}
-\frac{27\left(\delta^{\frac{m+3n-2}{3}}-1\right)}{(2m-3n-1)(m+3n-2)}\\
&-\frac{\delta^{m-1}-1}{m-1}
\left(\frac{1}{3n-1}+\frac{9}{m-6n+1}-\frac{9}{2m-3n-1}+\frac{1}{m-1}\right)\big\}\\
&>\frac{(m+3n-2)^{3}}{\text{e}},
\endaligned
\end{equation}
respectively.

$ Theorem$ $ 2.5$ and $Remark$ $2.2$ imply if both $(3.4)$ and $(3.5)$ hold, then $Eq.$ $ (3.2)$ is oscillatory.

$ Theorem$ $ 2.7$. Since condition $(2.19)$ is not satisfied, the related result from $ Theorem$ $ 2.7$
does not apply.

$ Theorem$ $ 2.8$ and $ Theorem$ $ 2.9$ in the sense that oscillation of $Eq.$ $ (3.2)$ is guaranteed
by the conditions $(3.3)$ and $(3.5)$.

\vspace{.4cm} \noindent{\bf 4. Summary}
 \setcounter{section}{4}
 \setcounter{equation}{0}

In  this paper, we studied the third-order  differential equation $(1.1)$ with noncanonical operators.
First, we established one-condition criteria for $property$ $A$ of $(1.1)$. Next, we presented various
two-condition criteria ensuring oscillation of all solutions of $(1.1)$. Finally, our results are applicable
on Euler-type equations of the forms $ (3.1)$ and $ (3.2)$. It remains open how to generalize these results
for higher-order noncanonical equations with deviating arguments.

\noindent {\bf Acknowledgements}\hspace{0.2cm}

The authors would like to express their highly appreciation to the editors and the referees
for their  valuable comments.

\vspace{15pt}

\noindent{\bf References}
\newcounter{cankao}
\begin{list}
{[\arabic{cankao}]}{\usecounter{cankao}\itemsep=-0.05cm}

\bibitem{1} G.E.Chatzarakis, J.D\v{z}urina, I.Jadlovsk\'{a}, New oscillation criteria for second-order
half-linear advanced differential equations, {\em  Appl. Math. Comput.}  347, (2019) 404-416.

\bibitem{2}  W.F.Trench, Canonical forms and principal systems for general disconjugate equations,
{\em  Trans. Amer. Math. Soc.}  189, (1973) 319-327.

\bibitem{3}  I.T.Kiguradze, T.A.Chanturia, Asymptotic properties of solutions of nonautonomous ordinary
differential equations, {\em Mathematics and its Applications (Soviet Series), }  vol. 89,
Kluwer Academic Publishers Group, Dordrecht, 1993,
doi: 10.1007/978-94-011-1808-8. Translated from the 1985 Russian original.

\bibitem{4} G.S.Ladde, V.Lakshmikantham, B.G.Zhang, Oscillation Throry of Differential Equations with
Deviating Arguments, {\em Marcel Dekker,. Inc., } New York, 1987.

\bibitem{5} H.Wu, L.Erbe, A.Peterson, Oscillation of solution to second-order half-linear delay
dynamic equations on time scales, {\em Electron J. Diff. Equ. }  71, (2016) 1-15.

\bibitem{6}  J.D\v{z}urina, I.Jadlovsk\'{a}, Oscillation of third-order differential equations
with noncanonical operators, {\em  Appl. Math. Comput. }  336, (2018) 394-402.

\bibitem{7}  R.P.Agarwal, S.R.Grace, D.O'Regan, Oscillation Theory for Second Order Linear, Half-Linear,
Superlinear, Sublinear Dynamic Equations, {\em Kluwer Academic Publishers, Dordrecht. } (2002) .

\bibitem{8}  R.P.Agarwal, S.R.Grace, D.O'Regan, Oscillation Theory for Second Order Dynamic Equations,
{\em Series in Mathematical Analysis and Applications. } 5. Taylor $\&$ Francis, Ltd., London , (2003) .

\bibitem{9}  R.P.Agarwal, M.Bohner, W.-T.Li, Nonoscillation and Oscillation: Theory for Functional Differential
 Equations, {\em Monographs and Textbooks in Pure and Applied Mathematics. } 267 , Marcel Dekker,
 Inc., New York, (2004) .

\bibitem{10} R.P.Agarwal, S.R.Grace, D.O'Regan, Oscillation Theory for Difference and
Functional Differential Equations,, {\em  Springer Science $\&$ Business Media.}  (2013) .

\bibitem{11} R.P.Agarwal, C.Zhang, T.Li, New Kamenev-type oscillation criteria for second-order
nonlinear advanced dynamic equations, {\em  Appl. Math. Comput. } 225, (2013) 822-828.

\bibitem{12} R.P.Agarwal, C.Zhang, T.Li, Some remarks on oscillation of second order neutral differential
equations, {\em  Appl. Math. Comput. } 274, (2016) 178-181.

\bibitem{13} E.M.Elabbasy, T.S.Hassan, B.M.Elmatary, Oscillation  criteria for third order delay nonlinear
differential equations, {\em Electron. J. Qual. Theory Differ. Equ. } 2012 (5) (2012) 1-11.

\bibitem{14} M.Bohner, S.R.Grace, I.Jadlovsk\'{a}, Oscillation criteria for third-order functional differential
equations with damping, {\em Electron. J. Differ. Equ. } 2016 (215) (2016)  1-15.

\bibitem{15} T.Li, C.Zhang, G.Xing, Oscillation of third-order neutral delay differential equations,
{\em Abstr. Appl. Anal. } 2012  (2012)  1-11.

\bibitem{16} R.Agarwal, M.Bohner, T.Li, C.Zhang, et. al., Oscillation of third-order nonlinear delay differential
equations, Taiwan, {\em . J. Math. } 17 (2) (2013) 545-558.

\bibitem{17} R.P.Agarwal, M.F.Aktas, A.Tiryaki, On oscillation criteria for third order nonlinear delay differential
equations, {\em  Arch. Math. } (Brno) 45 (1) (2009) 1-18.

\bibitem{18} M.Aktas, A.Tiryaki, A.Zafer, Oscillation criteria for third-order nonlinear functional differential
equations, {\em Appl. Math. Lett. } 23 (7) (2010) 756-762.

\bibitem{19} B.Bacul\'{i}kov\'{a}, J.D\v{z}urina, Oscillation of third-order functional differential equations,
{\em Electron. J. Qual. Theory Differ. Equ. } (43) (2010) 1-10.

\bibitem{20} B.Bacul\'{i}kov\'{a}, E.M.Elabbasy, S.H.Saker, J.D\v{z}urina, Oscillation criteria for third-order
nonlinear differential equations,
{\em  Math. Slovaca } 58 (2) (2008) 201-220, doi:10.2478/s12175-008-0068-1.

\bibitem{21} T.Candan, R.S.Dahiya, Oscillation of third order functional differential equations with delay,
in: Proceedings of the Fifth
Mississippi State Conference on Differential Equations and Computational Simulations (Mississippi State, MS, 2001),
 in: {\em Electron. J.  Differ. Equ. Conf.,}
 vol. 10, Southwest Texas State Univ., San Marcos, TX, 2003, pp. 79-88.

\bibitem{22} M.Cecchi, Z.Do\v{s}l\'{a}, M.Marini, Some properties of third order differential operators,
{\em Czechoslov. Math. J.}  47 (4) (1997) 729-748.

\bibitem{23} M.Cecchi, Z.Do\v{s}l\'{a}, M.Marini, Disconjugate operators and  related differential equations,
in: Proceedings of the Sixth Colloquium on the Qualitative Theory of Differential Equations, Szeged,
2000, p.No.4.17, {\em Electron. J. Qual. Theory Differ. Equ. }

\bibitem{24} G.E.Chatzarakis,  S.R.Grace, I.Jadlovsk\'{a}, Oscillation criteria for third-order delay
differential equations, {\em Adv. Differ. Equ. } 2017 (1) (2017) 330.

\bibitem{25} J.D\v{z}urina, I.Jadlovsk\'{a}, A note on Oscillation  of second-order delay differential
equations, {\em Appl. Math. Lett. } 69 (2017) 126-132.

\bibitem{26} S.R.Grace,  R.P.Agarwal, R.Pavani, E. Thandapani, On the oscillation of certain third order
nonlinear functional differential equations, {\em Appl. Math. Comput. } 202 (1) (2008) 102-112,
doi: 10.1016/j.amc.2008.01.025.

\bibitem{27} J.K.Hale, Functional Differential Equations,   {\em Applied Mathematical Sciences, } 3,
Spring-Verlag New York, New York-Heidelberg, 1971.

\bibitem{28} C.Zhang, T.Li, B.Sun, E.Thandapani, On the  oscillation of higher-order half-linear
delay differential equations, {\em Appl. Math. Lett. } 24 (9) (2011).

\end{list}
\end{document}